\documentclass[11pt,reqno]{amsart}
\usepackage{fullpage,amsfonts,amsmath,amssymb}
\usepackage{amsfonts,amsmath,amssymb,fullpage,graphicx}
\usepackage{verbatim}
\usepackage{hyperref}

\newcommand{\what}{\widehat}%
\newcommand{\wtilde}{\widetilde}%
\newcommand{\R}{\mathbb R}%
\newcommand{\C}{\mathbb C}%
\newcommand{\N}{\mathbb N}%
\newtheorem{theorem}{Theorem}[section]

\newtheorem{proposition}[theorem]{Proposition}
\newtheorem{corollary}[theorem]{Corollary}

\theoremstyle{definition}
\newtheorem{definition}[theorem]{Definition}

\theoremstyle{definition}
\newtheorem{remark}[theorem]{Remark}

\numberwithin{equation}{section}

\begin{document}
\title{An analogue of Bochner's theorem for Damek-Ricci spaces}
\subjclass[2000]{Primary 43A85; Secondary 22E30} \keywords{Bochner's theorem, positive definite functions, radially positive definite functions}

\author{Sanjoy Pusti}
\address[Sanjoy Pusti]{Mathematics Research Unit; University of Luxembourg, Campus kirchberg; 6, rue Richard Coudenhove-Kalergi;
L-1359, Luxembourg.}
\email{sanjoy.pusti@uni.lu}

\begin{abstract}
We  characterize the image of radial  positive measures $\theta$'s on a harmonic $NA$ group  $S$ which satisfies $\int_S\phi_0(x)\,d\theta(x)<\infty$ under the spherical transform, where $\phi_0$ is the elementary spherical function.  
\end{abstract}
\maketitle
\section{Introduction}
A continuous function $f$ on $\mathbb R$ is said to be positive
definite if for any real numbers $x_1, \cdots, x_m$ and complex
numbers $\xi_1, \cdots, \xi_m$ the following holds
\begin{eqnarray*}
\sum_{k, j=1}^{m}f(x_k -x_j)\xi_k \overline{\xi_j} \geq 0.
\end{eqnarray*}
This condition is equivalent to$$\int_{\mathbb R}f(x)
 (\phi*\phi^{*})(x)\,dx \geq 0$$ for all $\phi\in
C_c^{\infty}(\mathbb R)$, where $\phi^\ast(x)=\overline{\phi(-x)}$. A celebrated theorem of S. Bochner states that a positive definite function is the Fourier transform of some finite positive measure on $\R$. Therefore it characterizes the image of finite positive measures under the Fourier transform. Also it gives an integral representation of the positive definite functions. This theorem has been extended to locally compact Abelian groups.  P. Graczyk and J.-J. L{\oe}b characterizes image of $K$-biinvariant finite positive measures on $G$ where $G$ is a connected, noncompact, complex semisimple Lie group with finite centre  and $K$ is a maximal compact subgroup of $G$ (see \cite{Gr}). 

 For any $NA$ group (also known as Damek-Ricci space) $S$ we consider  two types of problems.

\noindent(1) The first one is to get an integral representation of radially positive definite  functions on $S$. We prove that such functions are given by positive measures on $\R\cup i\R$, such that the measure on $\R$ is finite and the measure $\mu$ on $i\R$ satisfies $\int_\R e^{a|\lambda|}d\mu(\lambda)<\infty$ for all $a>0$. This is an analogue of M. G. Krein's theorem for evenly positive definite functions on $\R$ (see \cite[p. 196]{Gelfand}).

\noindent (2) In the second one we consider the problem to characterize the image of radial  positive measures $\theta$'s which satisfy $\int_S\phi_0(x)\,d\theta(x)<\infty$ under the spherical  transform. For this, first we show that the spherical transform of such measures are even, continuous, bounded functions on $\R$ and satisfy certain  positive definite like condition.  Conversely we prove that  any even, continuous, bounded  function on $\R$ which satisfies such positive definite like condition  is the spherical transform of some radial, positive measure $\theta$ on $S$. This measure $\theta$ satisfies $\int_S\phi_0(x)\, d\theta(x)<\infty$.  Then we prove that the image set is a subset of the positive definite functions on $\R$. This positive definite like condition can alternatively be stated as, the elements of the image space are positive linear functionals on the Banach algebra $\left(L^1(\R, |c(\lambda)|^{-2}d\lambda)_e, \odot\right)$, whereas in the classical Bochner's theorem the elements of the image space are positive linear functionals on the Banach algebra $(L^1(\R),\ast)$. The condition on the measure $\theta$ (i.e. $\int_S\phi_0(x)\,d\theta(x)<\infty$) is due to technical reason.  If a measure $\theta$ is finite then it satisfies $\int_S\phi_0(x)\,d\theta(x)<\infty$. We guess the image space for finite, positive, radial measures under the spherical transform and write it as a conjecture. 
We refer  \cite{B1,B2, Pus, AS1} for further study in this literature.

We are grateful to Professors J. Faraut, R. P. Sarkar and S. Thangavelu for many helpful discussions, comments and suggestions. 
\section{Preliminaries}
In this section we explain the required preliminaries for $NA$ groups, mainly from \cite{ADY}.
Let $\mathfrak n$ be a two step nilpotent Lie algebra with the inner product $\langle,\rangle$. Let $\mathfrak{z}$ be the centre of $\mathfrak{n}$ and $\mathfrak{p}$ be the orthogonal complement of $\mathfrak{z}$ in $\mathfrak{n}$. We call the Lie algebra $\mathfrak{n}$ an $H$-type algebra if for each $Z\in\mathfrak{z}$ the map $J_Z:\mathfrak{p}\rightarrow\mathfrak{p}$ defined by $$\langle J_ZX,Y\rangle=\langle [X,Y],Z\rangle$$ satisfies $J_Z^2=-|Z|^2I_{\mathfrak p}$, where $I_{\mathfrak p}$ is the identity operator on $\mathfrak{p}$. A connected, simply connected Lie group $N$ is called $H$-type group if its Lie algebra is an $H$-type Lie algebra. Since $\mathfrak{n}$ is nilpotent the exponential map is a diffeomorphism from $\mathfrak{n}$ to $N$. Therefore any element $n$ of $N$ can be expressed as $n=\exp(X+Z)$ for some $X\in \mathfrak{p}, Z\in \mathfrak{z}$. Hence we can parametrize elements of $N$ by the pairs $(X,Z), X\in \mathfrak{p}, Z\in \mathfrak{z}$. It follows from Campbell-Baker-Hausdorff formula that the group law on $N$ is given by $$(X, Z)(X',Z')=(X+X', Z+Z'+\frac 12[X,X']).$$

Let $A=\R^+=(0,\infty)$ and $N$ an $H$-type group. For $a\in A$ we define a dilation $\delta_a:N\rightarrow N$ by $\delta_a\left((X,Z)\right)=\left(a^{1/2}X, aZ\right)$. Let $S=NA$ be the semi direct product of $N$ and $A$ under the dilation. Thus the multiplication on $S$ is given by $$(X,Z,a)(X',Z',a')=\left(X+a^{1/2}X', Z+aZ'+\frac 12 a^{1/2}[X,X'], aa'\right).$$ Then $S$ is a solvable, connected and simply connected Lie group having Lie algebra $\mathfrak{s}=\mathfrak{p}\oplus\mathfrak{z}\oplus \R$ with Lie bracket $$[(X,Z,l),(X',Z',l')]=\left(\frac 12 lX'-\frac 12 l'X, lZ'-l'Z+[X,X'], 0\right).$$We note that for any $Z\in\mathfrak{z}$ with $|Z|=1$, we have $J_Z^2=-I_{\mathfrak p}$. Therefore $J_Z$ defines a complex structure on $\mathfrak{p}$. Hence $\dim\mathfrak{p}$ is even. Let $\dim\mathfrak{p}=m$ and dim$\mathfrak{z}=k$. Then $Q=\frac m2 +k$ is called the {\em homogenous dimension} of $S$. We also use the symbol $\rho$ for $\frac Q2$ and $n$ for $m+k+1=\dim \mathfrak{s}$.

The group $S$ is equipped with the left-invariant Riemannian metric  induced by $$\langle(X,Z,l), (X',Z',l')\rangle=\langle X,X'\rangle +\langle Z,Z'\rangle +ll'$$ on $\mathfrak{s}$. The associated left invariant Haar measure $dx$ on $S$ is given by $a^{-Q-1}dXdZda$ where $dX, dZ, da$ are the Lebesgue measures on $\mathfrak{p},\mathfrak{z}$ and $\R^+$ respectively. Also the following integral formula holds:\begin{equation}
\int_Sf(x)\,dx=C_1\int_\R\int_N f(na_t)e^{2\rho t}\,dt\,dn \label{Iwasawa}
\end{equation}
The constant  $C_1$ depends on the normalization of the Haar measures involved.

The group $S$ can also be realized as the unit ball $$B(\mathfrak s)=\{(X,Z,l)\in\mathfrak{s}\mid |X|^2 +|Z|^2+ l^2<1\}$$ via Caley transform $C:S\rightarrow B(\mathfrak s)$ (see \cite[(1.12)]{ADY} for details). For $x\in S$ we let $r(x)=d(C(x), 0)$. A function $f$ on $S$ is called {\em radial} if $f(x)=f(r(x))$ for all $x\in S$. For a suitable function $f$ on $S$, its {\em radial component} is defined by $$f^\sharp(x)=\frac{\Gamma(\frac n2)}{2\pi^{\frac n2}}\int_{\partial B(\mathfrak s)} f(r(x)\sigma)\,d\sigma \text{ for all } x\in S.$$It is clear that if $f$ is radial function, then its radial component $f^\sharp(x)=f(x)$ for all $x\in S$. The convolution between two radial functions $f,g$ on $S$ is given by $(f\ast g)(x)=\int_S f(y^{-1})g(yx)\,dy$. Then it is easy to check that $f\ast g$ is a radial function. Also for two suitable radial functions $f$ and $g$ we have $f\ast g=g\ast f$.

Let $\mathbb D(S)^\sharp$ be the algebra of invariant differential operators on $S$ which are radial i.e., the operators which commutes with the operator $f\mapsto f^\sharp$. Then $\mathbb D(S)^\sharp$ is a polynomial algebra with single generator, the Laplace-Beltrami operator $\mathcal L$. A function $\phi$ on $S$ is called {\em spherical function} if $\phi(e)=1$ and $\phi$ is a radial eigenfunction of $\mathcal L$. All spherical functions are given by (see \cite{ADY}) $$\phi_\lambda(x)=\left(a(x)^{\rho-i\lambda}\right)^\sharp, \lambda\in\C$$ where $a(x)=e^t$ if $x=ne^t$.
Then it follows that $$\mathcal L\phi_\lambda=-(\lambda^2+\rho^2)\phi_\lambda \text{ for all }\lambda\in\C.$$ Also $\phi_\lambda(x)=\phi_{-\lambda}(x)$, $\phi_\lambda(x)=\phi_\lambda(x^{-1})$ and $\phi_\lambda(e)=1$. The spherical functions satisfy the basic estimate (see \cite{ADY}):
$$\phi_{i(\frac 1p-\frac 12)Q}(x)\asymp \left\{\begin{array}{lll}
e^{-\frac{Q}{p'}r(x)} & \text{ if } & 1\leq p<2\\ 
(1+r(x))e^{-\rho r(x)} & \text{ if } & p=2.
\end{array}\right.$$
Here $A\asymp B$ means there exists positive constants $C_1, C_2$ such that $C_1B\leq A\leq C_2B$. Also we have $$|\phi_\lambda(x)|\leq \phi_{i\mu}(x) \text{ for }|\Im \lambda|\leq \mu$$ where $\Im \lambda$ denotes the imaginary part of $\lambda$ and $\phi_{i\rho}\equiv 1$. Therefore if $\lambda\in\R$ then $|\phi_\lambda(x)|\leq \phi_0(x)$ for all $x\in S$.

For a suitable radial function $f$ on $S$, its {\em spherical transform} is defined by $$\what{f}(\lambda)=\int_S f(x)\phi_\lambda(x)\,dx.$$Then $\what{f\ast g}(\lambda)=\what{f}(\lambda)\what{g}(\lambda)$.

Let $C_c^\infty(S)^\sharp$ be the set of compactly supported radial $C^\infty$ functions on $S$. Also for $0<p\le 2$ the $L^p$-Schwartz space is defined by $$\mathcal C^p(S)^\sharp=\left\{f\in C^\infty(S)^\sharp\mid \sup_{x\in S}(1+r(x))^N\left|\mathcal L^t f(x)\right|e^{\frac 2p\rho r(x)}<\infty, \forall N, t\in \N\cup\{0\}\right\}.$$ We recall that $C_c^\infty(S)^\sharp$ is dense in $\mathcal C^p(S)^\sharp$.

For $R>0$  the Paley-Wiener space $PW_{R}(\mathbb C)$ is the set of
all entire functions $h:\mathbb C\rightarrow \mathbb C$ satisfying
for each $N \in \mathbb N$
$$|h(\lambda)|\leq C_{N}(1+|\lambda|)^{-N}e^{R|\Im\lambda|}\text {
for all } \lambda\in \C$$ for some constant $C_{N}>0$ depending on
$N$. Let $PW(\C)=\cup_{R>0}PW_R(\C)$. We shall denote the set of all even functions in
$PW(\mathbb C)$  by $PW(\mathbb C)_{e}$.

For $p\in (0, 2]$ we define $\gamma_p=(2/p-1)$. We consider the strip $S_p=\{z\in \C\mid |\Im
z|\leq \gamma_p\rho\}$ and  note that when $p=2$ then the strip becomes the line $\R$. For $0<p<2$ let $S_p^\circ$ and $\partial S_p$
respectively be the interior and boundary of the strip.

We define
$\mathcal S(S_p)$ to be the set of all functions
$h:S_p\rightarrow \mathbb C$ which are continuous on $S_p$,
holomorphic on $S_p^\circ$ (when $p=2$ then the function is simply
$C^\infty$ on $S_2=\R$) and satisfies $\sup_{\lambda\in
S_p}(1+|\lambda|^{r})|\frac{d^{m}}{d\lambda^{m}}h(\lambda)|<
\infty$, for all $r,m\in \mathbb N\cup\{0\}$. Let $\mathcal
S(S_p)_e$  denotes the subspaces of
$\mathcal S(S_p)$ consisting of even  functions. Topologized by the seminorms above it  can be
verified that $\mathcal S(S_p)$ and $\mathcal S(S_p)_e$  are Fr\'{e}chet spaces. Also let $C_c^\infty(\R)_e, L^1(\R,d\lambda)_e$ and $L^1(\R, |c(\lambda)|^{-2}d\lambda)_e$ be the subspaces of even functions of $C_c^\infty(\R)$, $L^1(\R,d\lambda)$ and $L^1(\R, |c(\lambda)|^{-2}d\lambda)$ respectively. They are equipped with the subspace topologies. From the basic estimates of $\phi_\lambda$ it follows that the domain of spherical transform of a function in $C_c^\infty(S)^\sharp$ is $\C$ and that of a function in $\mathcal  C^p(S)^\sharp$ is  $S_p$.
We have the following Paley-Wiener and $L^p$-Schwartz space isomorphism theorem (see \cite{ADY}).
\begin{theorem}\label{sch-iso-na}
 The map $f\mapsto
\widehat{f}$  is a topological isomorphism between
$C_c^\infty(S)^\sharp$ and $PW(\mathbb C)_{e}$ and also between $\mathcal  C^p(S)^\sharp$ and $\mathcal S(S_p)_e$.
\end{theorem}

For a suitable radial function $f$ on $S$, the Abel transform is defined by $$\mathcal A f(t)= e^{-\rho t}\int_N f(na_t)\,dn \text{ where } a_t=e^t.$$ It satisfies the relation $\what{f}(\lambda)=\wtilde{\mathcal A f}(\lambda)$ for a suitable radial function $f$ on $S$, where $\wtilde{\mathcal A f}$ is the Euclidean Fourier transform of $\mathcal A f$. Therefore it follows from Theorem \ref{sch-iso-na} that the Abel transform $f\mapsto \mathcal A f$ is a topological isomorphism between $C_c^\infty(S)^\sharp$ and $C_c^\infty(\R)_e$ and also between $\mathcal C^p(S)^\sharp$ and $\mathcal S(S_p)_e$.

\section{Bochner's theorem}
In this section first we give an integral representation of the {\em radially positive definite} functions. Then we characterize the image of radial positive measures $\theta$'s which satisfies $\int_G \phi_0(x)\,d\theta(x)<\infty$ under the spherical transform.

We have the following theorem (see \cite[Theorem 5, p. 226]{Gelfand}). 

\begin{theorem}
\label{evenly-multi-paley-space}
 Let $T$ be an evenly positive definite distribution on $\R$ i.e. $T(\phi\ast \phi^\ast)\geq 0$ for all $\phi\in C_c^\infty(\R)_e$. Then there exists even positive  measures $\mu_1$ and $\mu_2$ such that $$T(\phi)=\int_\R\wtilde{\phi}(\lambda)\,d\mu_1(\lambda) + \int_\R \wtilde{\phi}(i\lambda)\,d\mu_2(\lambda) \text{ for all } \phi\in C_c^\infty(\R)_e$$ where  $\mu_1$ is a tempered measure and $\mu_2$ is such that $\int_\R e^{a|\lambda|}\, d\mu_2(\lambda)<\infty$ for all $a>0$.
\end{theorem}
Here by a tempered measure  $\mu$ on $\R$ we mean that the measure $\mu$ satisfies $\int_\R\frac{1}{(1+|\lambda|^2)^p}\,d\mu(\lambda)<\infty$ for some $p>0$.

We call a continuous, radial function $f$ on $S$ {\em radially positive definite} if $$\int_S f(x)(g\ast g^\ast)(x)\,dx\geq 0, \text{ for all } g\in C_c^\infty(S)^\sharp\text{ where } g^\ast(x)=\overline{g(x^{-1})}.$$ If the equation above is true for every $g\in C_c^\infty(S)$,  we say that $f$ is a positive definite function. Then it is clear that the set of positive definite radial functions is a subset of the set of radially positive definite functions. 

\begin{theorem}
\label{boch-Fou-side}
 Let $f$ be a radially positive definite function on $S$. Then there exists even positive measures $\mu_1$ and $\mu_2$ such that $$f(x)=\int_\R\phi_\lambda(x)\, d\mu_1(\lambda) + \int_\R \phi_{i\lambda}(x)d\mu_2(\lambda) \text{ for all } x\in S$$ where $\mu_1$ is a finite measure and $\mu_2$ is such that $\int_\R e^{a|\lambda|}\, d\mu_2(\lambda)<\infty$ for all $a>0$.
\end{theorem}
\begin{proof}
We define a distribution $T_f: C_c^\infty(\R)_e\rightarrow \C$ by $$T_f(h)=\int_S f(x)\mathcal A^{-1}h(x)\,dx$$for all $h\in C_c^\infty(\R)_e$. By the continuity of $f$ and isomorphism of Abel transform it is easy to check that the integral exists and $T_f$ is continuous.  Also we can easily check that $\mathcal A^{-1}(h\ast h^\ast)=\mathcal A^{-1}h \ast (\mathcal A^{-1}h)^\ast$. Then since $f$ is radially positive definite function, for all $h\in C_c^\infty(\R)_e$  $$T_f(h\ast h^\ast)=\int_S f(x)\left(\mathcal A^{-1}h\ast (\mathcal A^{-1}h)^\ast\right)(x)\,dx\geq 0.$$ That is $T_f$ is an evenly positive definite distribution on $\R$. Hence by Theorem \ref{evenly-multi-paley-space} there exists even positive measures $\mu_1$ and $\mu_2$ such that 
$$T_f(h)=\int_\R \wtilde{h}(\lambda)\, d\mu_1(\lambda) + \int_\R \wtilde{h}(i\lambda)\, d\mu_2(\lambda)\text{ for all } h\in C_c^\infty(\R)_e$$ 
where $\mu_1$ is a tempered measure and $\mu_2$ is such that $\int_\R e^{a|\lambda|}d\mu_2(\lambda)<\infty$ for all $a>0$. Therefore we have $$\int_S f(x)g(x)\,dx=\int_\R \what{g}(\lambda)\, d\mu_1(\lambda) + \int_\R \what{g}(i\lambda)\, d\mu_2(\lambda)\text{ for all } g\in C_c^\infty(S)^\sharp.$$ Now we shall show that the measure $\mu_1$ is finite. For this let $\{\alpha_n\}$ be a $\delta$-sequence in $C_c^\infty(S)^\sharp$ and let $g_n=\alpha_n\ast \alpha_n^\ast$. Then $\{g_n\}$ is a $\delta$-sequence in $C_c^\infty(S)^\sharp$. Also, $\what{g_n}(\lambda)=|\what{\alpha_n}(\lambda)|^2\geq 0$ for all $\lambda\in\C$ and $\lim_{n\rightarrow\infty}\what{g_n}(\lambda)=\lim_{n\rightarrow\infty}\int_S g_n(x)\phi_\lambda(x)\,dx=\phi_\lambda(e)=1.$ Now the equation  $$\int_S f(x)g_n(x)\,dx=\int_\R \what{g_n}(\lambda)\, d\mu_1(\lambda) + \int_\R \what{g_n}(i\lambda)\, d\mu_2(\lambda)$$ implies that $\int_S f(x)g_n(x)\,dx\geq \int_\R \what{g_n}(\lambda)\, d\mu_1(\lambda)$ (since $\what{g_n}(\lambda)\geq 0$ for all $\lambda\in\C$ and measure $\mu_1, \mu_2$ are positive). Then using Fatou's lemma we get that $$f(e)=\lim_{n\rightarrow\infty}\int_S f(x)g_n(x)\,dx\geq \lim_{n\rightarrow\infty}\int_\R\what{g_n}(\lambda)\,d\mu_1(\lambda)\geq \int_\R\lim_{n\rightarrow\infty}\what{g_n}(\lambda)\,d\mu_1(\lambda).$$This shows that $$\int_\R\,d\mu_1(\lambda)\leq f(e).$$Therefore $\mu_1$ is finite measure. Then using Fubini's theorem we get that
$$\int_S f(x)g(x)\, dx = \int_S g(x)\int_\R \phi_\lambda(x)d\mu_1(\lambda) \,dx + \int_S g(x)\int_\R\phi_{i\lambda}(x)\, d\mu_2(\lambda)\,dx.$$
The equation above is true for every function $g\in C_c^\infty(S)^\sharp$. Hence $$f(x)=\int_\R\phi_\lambda(x)\, d\mu_1(\lambda) + \int_\R \phi_{i\lambda}(x)d\mu_2(\lambda) \text{ for all } x\in S$$ where $\mu_1$ is finite positive even measure and $\mu_2$ is positive even measure such that $\int_\R e^{a|\lambda|}\, d\mu_2(\lambda)<\infty$ for all $a>0$.
\end{proof}

Now we define an operation $\odot$ on the suitable functions on $\R$ which will make $L^1(\R, |c(\lambda)|^{-2}d\lambda)$ an algebra (cf. \cite{Flensted-Jensen}).

If $\lambda\in\R$ then $|\phi_\lambda(x)|\le \phi_0(x)\le (1+r(x))e^{-\rho r(x)}$ for all $x\in S$ . Also for Laplace-Beltrami operator $\mathcal L$ we have $$|\mathcal L \phi_\lambda (x)|\leq C (1+|\lambda|)^t\phi_0(x) \text{ for some constants }C>0, t\geq 0, \text{ for all }x\in S.$$This shows that the function $$f_{\lambda,\mu}:x\mapsto \phi_\lambda(x)\phi_\mu(x)\in \mathcal \mathcal \mathcal \mathcal \mathcal C^2(S)^\sharp \text{ for }\lambda, \mu\in\R.$$ Therefore the spherical  transform $$\what{f_{\lambda, \mu}}(\nu)=\int_S f_{\lambda,\mu}(x)\phi_\nu(x)\,dx=\int_S \phi_\lambda(x)\phi_\mu(x)\phi_\nu(x)\,dx $$exists on $\R$ and it belongs to $\mathcal S(\R)_e$. Let us denote $\what{f_{\lambda, \mu}}(\nu)$ by $K(\lambda,\mu,\nu)$ which exists for $\lambda,\mu,\nu\in\R$. Also it is easy to prove that for any even polynomial $q$, there exists constants $C, r_1, r_2 \geq 0$ such that $|q(\lambda)K(\lambda, \mu, \nu)|\leq C (1 +|\mu|)^{r_1}(1+|\nu|)^{r_2}$ for all $\lambda, \mu, \nu\in\R$.

Using the inversion theorem we get that 
$$\phi_\lambda(x)\phi_\mu(x)= c_0\int_\R K(\lambda,\mu,\nu)\phi_\nu(x)|c(\nu)|^{-2}\,d\nu \text{ where } c_0=2^{k-2}\pi^{-\frac n2-1}\Gamma(\frac n2).$$
Suppose $f, g\in \mathcal \mathcal \mathcal \mathcal \mathcal C^2(S)^\sharp$. Then by the inversion formula we have
 $$f(x)g(x)= c_0^3\int_\R\left(\int_\R\int_\R \what{f}(\lambda)\what{g}(\mu)K(\lambda,\mu,\nu)|c(\lambda)|^{-2}|c(\mu)|^{-2}\,d\lambda \,d\mu\right)\phi_\nu(x)|c(\nu)|^{-2}\,d\nu.$$
For $A, B\in \mathcal S(\R)$ we define $$A\odot B(\nu)=c_0^2\int_\R\int_\R A(\lambda)B(\mu)K(\lambda,\mu,\nu)|c(\lambda)|^{-2}|c(\mu)|^{-2}\,d\lambda\,d\mu.$$ Then for $A, B\in \mathcal S(\R)$, $A\odot B$ exists on $\R$ as the Plancherel measure $|c(\lambda)|^{-2}$ has polynomial growth (see \cite[(2.31)]{ADY}). Therefore it follows that $$f(x)g(x)=c_0\int_\R \left(\what{f}\odot \what{g}\right)(\nu)\phi_\nu(x)|c(\nu)|^{-2}d\nu.$$ This shows that 
\begin{equation}\label{newproduct}
\what{f.g}(\lambda)=\what{f}\odot\what{g}(\lambda).
\end{equation}
 Hence $A, B\in \mathcal S(\R)$ implies that $A\odot B\in\mathcal S(\R)$.  Also for suitable radial function $h$ on $S$, we have 
\begin{equation}\label{relation}
\int_S h(x)(f.g)(x)\,dx=\int_\R \what{h}(\lambda)(\what{f}\odot \what{g})(\lambda)|c(\lambda)|^{-2}\,d\lambda.
\end{equation}

From \cite[Theorem 4.4]{Flensted-Jensen} and \cite[(2.13), (2.14)]{ADY} it follows that $K(\lambda, \mu, \nu)=\what{f_{\lambda,\mu}}(\nu)\geq 0$ for all $\lambda, \mu,\nu\in\R$. From this fact it is easy to check that $L^1(\R, |c(\lambda)|^{-2}d\lambda)$ is an algebra under the operation $\odot$ and $\|f\odot g\|_{L^1(\R,|c(\lambda)|^{-2}d\lambda)}\leq \|f\|_{L^1(\R,|c(\lambda)|^{-2}d\lambda)} \|g\|_{L^1(\R,|c(\lambda)|^{-2}d\lambda)}$.

\begin{definition}
 For a  positive radial measure $\theta$ on $S$, its spherical transform is defined by $\what{\theta}(\lambda)=\int_S\phi_\lambda(x)\,d\theta(x)$, whenever the integral exists.
\end{definition}
\begin{proposition}\label{measure-group} Let $\theta$ be a radial positive measure on $S$ such that $\int_S\phi_0(x)\, d\theta(x)<\infty$. Then the following conditions are satisfied:
\begin{enumerate}
 \item[(1)]$\what{\theta}$ exists on $\R$ and it is an even, continuous, bounded function on $\R$; 
\item[(2)]$\int_\R \what{\theta}(\lambda)\left(g\odot g^\ast\right)(\lambda)|c(\lambda)|^{-2}\,d\lambda \geq 0 \text{ for all } g\in \mathcal S(\R)_e.$
\end{enumerate}
\end{proposition}
\begin{proof}
$(1)$ Existence and boundedness of $\what{\theta}$ follows from the fact that for $\lambda\in\R$, $|\phi_\lambda(x)|\leq \phi_0(x)$ for all $x\in S$. Also $\phi_\lambda(x)=\phi_{-\lambda}(x)$ for all $x\in S, \lambda\in\R$ implies that $\what{\theta}$ is an even function. Continuity follows from the  dominated convergence theorem.

$(2)$
 Let $\alpha\in \mathcal \mathcal \mathcal \mathcal \mathcal C^2(S)^\sharp$ be such that $\what{\alpha}(\lambda)=g(\lambda)$. Then 
$$\int_\R \what{\theta}(\lambda)\left(g\odot g^\ast\right)(\lambda)|c(\lambda)|^{-2}\,d\lambda= \int_\R \what{\theta}(\lambda)\left(\what{\alpha}\odot (\what{\alpha})^\ast\right)(\lambda)|c(\lambda)|^{-2}\,d\lambda.$$
This is equal to $\int_\R \what{\theta}(\lambda)\what{\left(\alpha\overline{\alpha}\right)}(\lambda)|c(\lambda)|^{-2}\,d\lambda$. Then using the definition and the Fubini's theorem we get that $$\int_\R \what{\theta}(\lambda)\left(g\odot g^\ast\right)(\lambda)|c(\lambda)|^{-2}\,d\lambda=\int_S \int_\R \what{\left(\alpha\overline{\alpha}\right)}(\lambda)\phi_\lambda(x)|c(\lambda)|^{-2}\,d\lambda \,d\theta(x).$$
Then using the inversion formula we get that
$$
\int_\R \what{\theta}(\lambda)\left(g\odot g^\ast\right)(\lambda)|c(\lambda)|^{-2}\,d\lambda = \frac{1}{c_0}\int_S \left(\alpha.\overline{\alpha}\right)(x)\,d\theta(x)= \frac{1}{c_0}\int_S |\alpha(x)|^2 \,d\theta(x)\geq 0.
$$
\end{proof}

In the following theorem we characterize the image of such radial positive measures under the spherical transform. 
\begin{theorem}
\label{boch-group}
An even, continuous, bounded function $h$ on $\R$ is the spherical transform of a radial  positive measure $\theta$ on $S$ which satisfies $\int_S \phi_0(x)\,d\theta(x)<\infty$ if and only if $h$ satisfies the  condition  $\int_\R h(\lambda)(g\odot g^\ast)(\lambda)|c(\lambda)|^{-2}\,d\lambda\geq 0$ for all $g\in \mathcal S(\R)_e$.
\end{theorem}

The classical Bochner's theorem can be restated as follows: A continuous function $p$ on $\R$ is the Fourier transform of a finite positive measure on $\R$ if and only if it is a positive linear functional on the algebra $\left(L^1(\R,d\lambda), \ast\right)$.

Let $p$ be an even, continuous, bounded function on $\R$ which is a positive linear functional on the algebra $\left(L^1(\R,d\lambda)_e, \ast\right)$ i.e. $p$ satisfies $\int_\R p(\lambda) (l\ast l^\ast)(\lambda)\geq 0$ for all $l\in L^1(\R,d\lambda)_e$. Then by Krein's theorem (\cite{Gelfand}) there exists finite, positive, even measures $\nu_1, \nu_2$ on $\R$ such that $$p(\lambda)=\int_\R e^{i\lambda x}\,d\nu_1(x) + \int_{i\R} e^{i\lambda y}\,d\nu_2(y).$$ But the boundedness of $p$ implies that $p(\lambda)=\int_\R e^{i\lambda x}\,d\nu_1(x)$. Therefore $p$ is a positive definite function. Conversely any even positive definite function can be considered as a positive linear functional on the algebra $\left(L^1(\R,d\lambda)_e, \ast\right)$. 

Hence the Bochner's theorem for even functions on $\R$ can be stated as follows: An even, continuous, bounded function $p$ on $\R$ is the Fourier transform of a finite, positive, even measure on $\R$ if and only if it is a positive linear functional on the algebra $\left(L^1(\R,d\lambda)_e, \ast\right)$.

Also we can restate our theorem (Theorem \ref{boch-group})  alternatively as follows: An even, continuous, bounded function $h$ on $\R$ is the spherical transform of a radial  positive measure $\theta$ on $S$ which satisfies $\int_S \phi_0(x)\,d\theta(x)<\infty$ if and only if it is a positive linear functional on the Banach algebra $\left(L^1(\R, |c(\lambda)|^{-2}d\lambda)_e, \odot\right)$.

Therefore our theorem is  analogous to the classical Bochner's theorem.

We state the following corollary of the theorem:
\begin{corollary}
\label{image-subset-positive-definite}
 An even, continuous, bounded function $h$ on $\R$ which satisfies for all $g\in \mathcal S(\R)_e$, $$\int_\R h(\lambda)(g\odot g^\ast)(\lambda)|c(\lambda)|^{-2}\,d\lambda\geq 0$$  is a positive definite function on $\R$.
\end{corollary}

\begin{proof}[Proof of Corollary]
 Suppose $h$ is an even, continuous, bounded function on $\R$ which satisfies the condition of the theorem above. Then by the Theorem \ref{boch-group} there exists a radial positive measure $\theta$ on $S$ such that $\what{\theta}(\lambda)=h(\lambda)$ for all $\lambda\in \R$. The measure $\theta$ satisfies $\int_S\phi_0(x)\,d\theta(x)<\infty$. We consider the Abel transform $\mathcal A\theta$ of $\theta$ defined by $$d\mathcal A\theta(a_t)= e^{-\rho t}\int_N \, d\theta(na_t) \text{ where } a_t=e^t.$$Then $\mathcal A\theta$ is a measure on $\R$ having the property  that 
\begin{equation}\label{triangle}
 \wtilde{\mathcal A\theta}(\lambda):=\int_\R e^{-i\lambda t}\,d\mathcal A\theta(a_t)=\what{\theta}(\lambda).                                                                                                                                    \end{equation}
 From the condition $\int_S\phi_0(x)\,d\theta(x)<\infty$ and equation (\ref{triangle}) it follows that the measure $\mathcal A\theta$ is finite and hence its Fourier transform $\wtilde{\mathcal A\theta}$ is a positive definite function on $\R$. But $\wtilde{\mathcal A\theta}(\lambda)=\what{\theta}(\lambda)=h(\lambda)$. Therefore  $h$ is a positive definite function on $\R$.
\end{proof}

Let $\mathcal P_0$ be the set of all even, continuous, bounded functions $h$ on $\R$ such that for all $g\in \mathcal S(\R)_e$ $$\int_\R h(\lambda)(g\odot g^\ast)(\lambda)|c(\lambda)|^{-2}\,d\lambda\geq 0.$$ Also let $\mathcal P$ be the set of all positive definite functions on $\R$. Then using the theorem above we have the following partial informations about the set $\mathcal P_0$:
\begin{enumerate}
\item By the corollary above we have $\mathcal P_0\subseteq \mathcal P$.
\item For each fixed $x_0\in S$, the function $\lambda\mapsto \phi_\lambda(x_0)\in \mathcal P_0$. In particular $\mathcal P_0$ contains positive constants. Also for each fixed $x_0\in S$, the function $\lambda\mapsto \phi_\lambda(x_0)$ is a positive definite function on $\R$, which can also be concluded from its integral representation.
\item Let us consider the heat kernel $p_t$ on $S$. It is a radial, nonnegative function on $S$ such that $\what{p_t}(\lambda)=e^{-t(\rho^2+ \lambda^2)}$ (see \cite{ADY}). This fact together with the theorem above implies that for each $t>0$, the function $\lambda\mapsto e^{-t(\rho^2 + \lambda^2)}\in\mathcal P_0$.
\item If $\beta_1, \beta_2\in \mathcal P_0$ then it follows that $\beta_1 +\beta_2, \beta_1\beta_2, c\beta_1\in\mathcal P_0$ for any positive constant $c$.
\end{enumerate}

\begin{remark} The condition of the Theorem \ref{boch-group} on the measure i.e. $\int_G \phi_0(x)\,d\theta(x)<\infty$ is due to technical reason. A finite measure $\theta$ always satisfies the condition $\int_G \phi_0(x)\,d\theta(x)<\infty$, since $|\phi_0(x)|\leq 1$ for all $x\in S$. Spherical transform of a finite positive measure exists on $S_1$ as $\phi_\lambda\in L^\infty$ if and only if $\lambda\in S_1$. Then it is easy to prove that the spherical transform $\what{\theta}$ is  analytic on $S_1^0$, continuous on $S_1$ and satisfies the positive definite like condition (as stated in the Theorem \ref{boch-group}). For the converse we conjecture the following:

An even, bounded function $h$, which is analytic on the interior of $S_1$ and continuous on $S_1$ is the spherical transform of a radial, finite, positive measure $\theta$ on $S$ if and only if for all $g\in \mathcal S(\R)_e$ $$\int_\R h(\lambda)(g\odot g^\ast)(\lambda)|c(\lambda)|^{-2}\,d\lambda\geq 0.$$ 
\end{remark}

\begin{proof}[Proof of the Theorem \ref{boch-group}]
The necessity of the condition is proved in Proposition \ref{measure-group}.

For the sufficiency  we let $h$ be an even, continuous, bounded function on $\R$ which satisfies the condition above. We define a  linear functional $T:\mathcal S(\R)_e\rightarrow \C$ by $$T(g)=c_0\int_\R h(\lambda)g(\lambda)|c(\lambda)|^{-2}\, d\lambda \text{ for all } g\in \mathcal S(\R)_e.$$
 This linear functional exists and continuous by the boundedness of $h$. Using this we also define a continuous linear functional $\wtilde{T}:\mathcal \mathcal \mathcal \mathcal \mathcal C^2(S)^\sharp\rightarrow \C$ by $$\wtilde{T}(f)=T(\what{f}), \text{ for all } f\in \mathcal \mathcal \mathcal \mathcal \mathcal C^2(S)^\sharp.$$This linear functional is well defined and continuous by the Schwartz space isomorphism theorem (see Theorem  \ref{sch-iso-na}).
From the hypothesis we have $$T(g\odot g^\ast)\geq 0 \text{ for all } g\in \mathcal S(\R)_e.$$ This implies that $T(\what{\alpha}\odot (\what{\alpha})^\ast)\geq 0$ for all $\alpha\in \mathcal \mathcal \mathcal \mathcal \mathcal C^2(S)^\sharp$. That is $T(\what{\alpha\overline{\alpha}})\geq 0$ by equation (\ref{newproduct}), since  $\what{\overline{\alpha}}(\lambda)=(\what{\alpha})^\ast(\lambda)$. This condition is equivalent to $$\wtilde{T}(\alpha\overline{\alpha})\geq 0 \text{ for all } \alpha\in \mathcal \mathcal \mathcal \mathcal \mathcal C^2(S)^\sharp.$$

We claim that:  $$\wtilde{T}(\alpha)\geq 0 \text{ for all }\alpha\in \mathcal \mathcal \mathcal \mathcal \mathcal C^2(S)^\sharp \text{ with }\alpha\geq 0.$$ 
To prove the claim, we first show that $\{\alpha\overline{\alpha}\mid \alpha\in C_c^\infty(S)^\sharp\}$ is dense in $\{\alpha\geq 0\mid \alpha\in C_c^\infty(S)^\sharp\}$. For this we let $\psi$ be a positive function in $C_c^\infty(S)^\sharp$ and suppose $\psi(x)=0$ for $r(x)>a$. Let $\gamma$ be a compactly supported $C^\infty$ function on $\R$ with $\gamma(t)=1$ for $|t|\le a$. We extend $\gamma$ as a radial function to $S$. We define $$\psi_m(x)=\gamma(x)\sqrt{\psi(x)+\frac 1m}.$$Then $\psi_m\in C_c^\infty(S)^\sharp$ and $$\psi_m^2(x)=\psi_m(x)\overline{\psi_m(x)}=\gamma(x)^2\left(\psi(x)+\frac 1m\right)\rightarrow \psi(x)$$ in the topology of $C_c^\infty(S)^\sharp$. Therefore $\wtilde{T}(\alpha)\geq 0$ for all $\alpha\in C_c^\infty(S)^\sharp$ with $\alpha\geq 0$. Now we let $\alpha\in \mathcal \mathcal \mathcal \mathcal \mathcal C^2(S)^\sharp$ such that $\alpha\geq 0$. Then there exists a sequence $\alpha_n\in C_c^\infty(S)^\sharp$ with $\alpha_n\geq 0$ such that $\alpha_n\rightarrow \alpha$ in $\mathcal \mathcal \mathcal \mathcal \mathcal C^2(S)^\sharp$. Since each $\wtilde{T}(\alpha_n)\geq 0$ it follows that $\wtilde{T}(\alpha)\geq 0$. Hence the claim is established.

Therefore $\alpha\mapsto \wtilde{T}(\alpha)$ is a positive linear functional on $\mathcal \mathcal \mathcal \mathcal \mathcal C^2(S)^\sharp$. By Riesz representation theorem there is a radial positive measure $\theta$ on $S$ such that 
 $$\wtilde{T}(\alpha)=\int_S \alpha(x)\,d\theta(x) \text{ for all }\alpha\in C_c^\infty(S)^\sharp.$$
That is $$c_0\int_\R h(\lambda)\what{\alpha}(\lambda)|c(\lambda)|^{-2}\,d\lambda=\int_S\alpha(x)\,d\theta(x)\text{ for all }\alpha\in C_c^\infty(S)^\sharp.$$
But  $h$ is such that the linear functional $\alpha\mapsto \int_\R h(\lambda)\what{\alpha}(\lambda)|c(\lambda)|^{-2}\,d\lambda$ extends to $\mathcal \mathcal \mathcal \mathcal \mathcal C^2(S)^\sharp$. Therefore 
$$c_0\int_\R h(\lambda)\what{\alpha}(\lambda)|c(\lambda)|^{-2}\,d\lambda=\int_S\alpha(x)\,d\theta(x)\text{ for all }\alpha\in \mathcal \mathcal \mathcal \mathcal \mathcal C^2(S)^\sharp.$$
That is 
\begin{equation}\label{Riesz}
\int_\R h(\lambda)\what{\alpha}(\lambda)|c(\lambda)|^{-2}\,d\lambda=\int_S\left(\int_\R\what{\alpha}(\lambda)\phi_\lambda(x)|c(\lambda)|^{-2}\,d\lambda\right)\,d\theta(x) \text{ for all } \alpha\in \mathcal \mathcal \mathcal \mathcal \mathcal C^2(S)^\sharp.
\end{equation}

We shall show that the measure $\theta$ satisfies $\int_S\phi_0(x)\, d\theta(x)<\infty$. For that we consider the heat kernel $$p_t(x)=\int_\R e^{-t(\lambda^2 + \rho^2)}\phi_\lambda(x)|c(\lambda)|^{-2}\,d\lambda.$$ This is a radial function on $S$ satisfies $p_t(x)\ge 0$ for all $x\in S$ and $\int_Sp_t(x)\,dx=1$. Let us define $\gamma_n(\lambda)=\frac{1}{p_n(e)}e^{-n(\lambda^2 + \rho^2)}$. Then it follows that $\int_S \frac{p_n(x)}{p_n(e)}\phi_\lambda(x)\,dx= \gamma_n(\lambda)$ and $\int_\R\gamma_n(\lambda)|c(\lambda)|^{-2}\, d\lambda=1$. Now our claim is that for any $\beta>0$, $$\int_{|\lambda|\ge\beta}\gamma_n(\lambda)|c(\lambda)|^{-2}\,d\lambda \rightarrow 0 \text{ as } n\rightarrow\infty.$$ Let $\beta>0$ fixed and choose $\alpha>0$ such that $\beta>\alpha$. We have $p_n(e)=\int_\R e^{-n(\lambda^2 +\rho^2)}|c(\lambda)|^{-2}\,d\lambda$. Then $$\begin{array}{lll}
p_n(e) &\ge &\int_{-\alpha}^{\alpha}e^{-n(\lambda^2 + \rho^2)}|c(\lambda)|^{-2}\,d\lambda \\ \\
& \ge & e^{-n(\alpha^2 +\rho^2)}\int_{-\alpha}^{\alpha}|c(\lambda)|^{-2}\,d\lambda.                                                                                                                                                                                                                                                                                                                                                                                                                                                                                                                                                                                                                                                                                                                                                                                
                                                                                                                                                                                                                                                                                                                                                                                                                                                                                                                                                                                                                                                                                                                                                                                                                                             \end{array}
$$  
This implies that $p_n(e)\ge C_\alpha \,e^{-n(\alpha^2 +\rho^2)}$ where $C_\alpha$ is a positive constant depends only on $\alpha$. Now $$\int_{|\lambda|\ge\beta}e^{-n(\lambda^2+\rho^2)}|c(\lambda)|^{-2}\,d\lambda\leq e^{-(n-1)(\beta^2 + \rho^2)}\int_\R e^{-(\lambda^2 +\rho^2)}|c(\lambda)|^{-2}\,d\lambda = D_\beta\, e^{-n(\beta^2 + \rho^2)}$$ where $D_\beta$ is a positive constant depends only $\beta$. Therefore $$\int_{|\lambda|\geq \beta}\gamma_n(\lambda)|c(\lambda)|^{-2}\,d\lambda\leq A e^{-n(\beta^2-\alpha^2)}$$ where $A$ is a positive constant depends only on $\beta, \alpha$. This establishes the claim.

Therefore since $h$ is continuous and bounded we have $$\lim_{n\rightarrow \infty}\int_\R h(\lambda)\gamma_n(\lambda)|c(\lambda)|^{-2}\,d\lambda=h(0).$$ We apply the  sequence $\{\gamma_n\}$ to equation (\ref{Riesz}) and take limit $n\rightarrow\infty$ and use Fatou's lemma  to get 
$$h(0)=\lim_{n\rightarrow \infty}\int_\R h(\lambda) \gamma_n(\lambda)|c(\lambda)|^{-2}\,d\lambda=\lim_{n\rightarrow\infty} \int_S \frac{p_n(x)}{p_n(e)}\,d\theta(x)\ge \int_S \lim_{n\rightarrow \infty}\frac{p_n(x)}{p_n(e)}\,d\theta(x).$$  Now 
$$\begin{array}{lll}
   \lim_{n\rightarrow\infty}\frac{p_n(x)}{p_n(e)}&=& \lim_{n\rightarrow\infty}\int_\R \frac{e^{-n(\lambda^2 +\rho^2)}}{p_n(e)}\phi_\lambda(x)|c(\lambda)|^{-2}\,d\lambda \\ \\
&=& \lim_{n\rightarrow\infty}\int_\R \gamma_n(\lambda)\phi_\lambda(x)|c(\lambda)|^{-2}\,d\lambda\\ \\
&=& \phi_0(x).
  \end{array}
$$
Therefore we get $$\int_S\phi_0(x)\,d\theta(x)\le h(0).$$

From equation (\ref{Riesz}) using Fubini's theorem we get, 
$$\int_\R h(\lambda)\what{\alpha}(\lambda)|c(\lambda)|^{-2}\,d\lambda=\int_\R \what{\alpha}(\lambda)|c(\lambda)|^{-2}\left(\int_S \phi_\lambda(x)\, d\theta(x)\right)\, d\lambda.
$$
But the equation above is true for every $\alpha\in \mathcal  C^2(S)^\sharp$. This implies that $$h(\lambda)=\int_S \phi_\lambda(x)\,d\theta(x)\text{ for all }\lambda\in\R.$$ This completes the proof.
\end{proof}

\begin{remark}
\begin{enumerate}
 \item Let $\theta$ be a finite, positive, radial measure on $S$. Then its spherical transform $\what{\theta}$ is obviously analytic on $S_1^0$, continuous on $S_1$ and satisfies the positive definite like condition (as stated in Theorem \ref{boch-group}). Conversely if we start with an even function which is analytic on $S_1^0$, continuous on $S_1$ and satisfies the positive definite like condition we can proceed as in the proof of the theorem to get a measure $\theta$ which satisfies the equation (\ref{Riesz}) but from this we are unable to prove that the measure $\theta$ is finite. If we could prove that the measure $\theta$ is finite then we would get $h(\lambda)=\int_S \phi_\lambda(x)\,d\theta(x)\text{ for all }\lambda\in\R.$ From analyticity and continuity the equality would hold on the strip $S_1$.

\item A Riemannian symmetric space $X$ of noncompact type can be
realized as a quotient space $G/K$ where $G$ is a connected
noncompact semisimple Lie group with finite centre and $K$ is a
maximal compact subgroup of $G$. Also a symmetric space $X$ is an $NA$ group and radial functions of that $NA$ group are $K$-biinvariant functions on $G$. Therefore the theorems proved in this article for radial functions on $NA$ group is also true for $K$-biinvariant functions on real rank one noncompact, connected, semisimple Lie group $G$ with finite centre. The Theorem \ref{boch-group} is new in the real rank one symmetric space case also.
\end{enumerate}
\end{remark}

\end{document}